\documentclass[12pt,reqno]{amsart}

\newcommand\version{November 6, 2022}

\usepackage[T1]{fontenc}
\usepackage[utf8]{inputenc}
\usepackage{amsmath,amsfonts,amsthm,amssymb,amsxtra}
\usepackage{bbm} 
\usepackage[parfill]{parskip}    

\usepackage[ngerman, german, english]{babel} 

\newtheorem{theorem}{Theorem}
\newtheorem{proposition}[theorem]{Proposition}

\theoremstyle{definition}

\theoremstyle{remark}

\numberwithin{equation}{section}

\renewcommand{\epsilon}{\varepsilon}

\newcommand{\R}{\mathbb{R}}

\newcommand{\eps}{\varepsilon}
\newcommand{\xl}{{x, \lambda}}

\newcommand{\pxi}{\partial_{x_i}}

\newcommand{\Sph}{{\mathbb{S}}}

\numberwithin{equation}{section}

\DeclareMathOperator{\dist}{dist}

\begin{document}

\title{On the sharp constant in the Bianchi--Egnell stability inequality}

\author{Tobias K\"onig}
\address[Tobias K\"onig]{Institut für Mathematik, 
Goethe-Universität Frankfurt, 
Robert-Mayer-Str. 10, 60325 Frankfurt am Main, Germany}
\email{koenig@mathematik.uni-frankfurt.de}

\date{\version}
\thanks{\copyright\, 2022 by the author. This paper may be reproduced, in its entirety, for non-commercial purposes.\\
The author thanks R.~L.~Frank for helpful advice.}

\maketitle

\begin{abstract}
This note is concerned with the Bianchi-Egnell inequality, which quantifies the stability of the Sobolev inequality, and its generalization to fractional exponents $s \in (0, \frac{d}{2})$. 
We prove that in dimension $d \geq 2$ the best constant 
\[ c_{BE}(s) = \inf_{f \in \dot{H}^s(\mathbb R^d) \setminus \mathcal M} \frac{\|(-\Delta)^{s/2} f\|_{L^2(\mathbb R^d)}^2 - S_{d,s} \|f\|_{L^{2^*}(\mathbb R^d)}^2}{\dist_{\dot{H}^s(\mathbb R^d)}(f, \mathcal M)^2} \]
is strictly smaller than the spectral gap constant $\frac{4s}{d+2s+2}$ associated to sequences which converge to the manifold $\mathcal M$ of Sobolev optimizers. In particular, $c_{BE}(s)$ cannot be asymptotically attained by such sequences.  Our proof relies on a precise expansion of the Bianchi-Egnell quotient along a well-chosen sequence of test functions converging to $\mathcal M$. 
\end{abstract}

\section{Introduction and main result }

For $0 < s < \frac{d}{2}$, the (fractional) Sobolev inequality states that 
\begin{equation}
\label{Sobolev}
\|(-\Delta)^{s/2} f\|_{L^2(\R^d)}^2 \geq S_{d,s} \|f\|_{L^{2^*}(\R^d)}^2, 
\end{equation} 
for every $f$ in the homogeneous Sobolev space $\dot{H}^s(\R^d)$. Here 
\[ 2^* = \frac{2d}{d-2s} \]
is the critical Sobolev exponent and the best constant $S_{d,s}$ is given by
\begin{equation}
\label{S ds definition}
S_{d,s} = 2^{2s} \pi^s \frac{\Gamma(\frac{d+2s}{2})}{\Gamma(\frac{d-2s}{2})} \left( \frac{\Gamma(\frac{d}{2})}{\Gamma(d)} \right)^{\frac{2s}{d}} 
\end{equation} 
It was shown by Lieb \cite{Li} in an equivalent dual formulation (and before by Aubin \cite{Au} and Talenti \cite{Ta} for $s = 1$) that $S_{d,s}$  is  optimal  and that it is achieved precisely by the so-called \emph{Talenti bubbles}, i.e. functions belonging to the $(d+2)$-dimensional manifold
\begin{equation}
\label{M definition}
\mathcal M := \left \{ x \mapsto c (a + |x-b|^2)^{-\frac{d-2s}{2}} \, : \, a > 0, \, b \in \R^d, \, c \in \R \setminus \{0\} \right \}. 
\end{equation}

In recent years, the study of the stability properties of the Sobolev and other geometric inequalities has attracted much interest and remarkable advances. The natural stability inequality associated to \eqref{Sobolev} has been established by Bianchi and Egnell \cite{BiEg} for $s=1$ and $d \geq 3$ and was later extended in \cite{ChFrWe} to all $s \in (0, \frac{d}{2})$. It states that there is a constant $c_{BE}(s) > 0$ such that 
\begin{equation}
\label{bianchi-egnell}
\mathcal E(f) := \frac{\|(-\Delta)^{s/2} f\|_2^2 - S_{d,s} \|f\|_{2^*}^2}{\dist(f, \mathcal M)^2} \geq c_{BE}(s) \qquad \text{ for all } f \in \dot{H}^s(\R^d) \setminus \mathcal M. 
\end{equation}
Here and in the rest of this paper, $\dist$ denotes the distance in $\dot{H}^s(\R^d)$, i.e. $\dist (f, \mathcal M) := \inf_{g \in \mathcal M} \| (-\Delta)^{s/2} (f-g)\|_2$. Moreover, we abbreviate $\|\cdot\|_p  = \|\cdot\|_{L^p(\R^d)}$ for ease of notation.

Since the proof of \eqref{bianchi-egnell} in \cite{BiEg, ChFrWe} proceeds by compactness, it yields no explicit lower bound on the constant $c_{BE}(s)$. For $s=1$, the first constructive lower bound on $c_{BE}(1)$ is proved in the recent preprint \cite{DoEsFiFrLo}. 

On the other hand, it is standard to derive an upper bound on $c_{BE}(s)$ by using an explicit sequence of test functions of the form
\[ f_\eps(x) = (1 + |x|^2)^{-\frac{d-2s}{2}} + \eps \rho \]
converging to $\mathcal M$ as $\eps \to 0$. 
As observed in \cite{ChFrWe}, a suitable choice of $\rho$ then yields the bound
\begin{equation}
\label{spectral bound}
c_{BE}(s) \leq \frac{4s}{d+2s + 2} . 
\end{equation}

The purpose of this note is to prove that the inequality \eqref{spectral bound} must  in fact be strict.

\begin{theorem}
\label{theorem strict inequality}
Let $s \in (0, \frac{d}{2})$ with $d \geq 2$ and let $c_{BE}(s)$ be the optimal constant in \eqref{bianchi-egnell}. Then 
\[ c_{BE}(s) < \frac{4s}{d +2s + 2}. \]
\end{theorem}

Let us emphasize that Theorem \ref{theorem strict inequality} covers in particular the classical case $s = 1$, $d \geq 3$, and that its conclusion is new also for this case. 

Theorem \ref{theorem strict inequality} shows that the study of the sharp constant $c_{BE}(s)$ cannot be reduced to a local analysis near the manifold $\mathcal M$. This phenomenon is analogous to the situation for the planar isoperimetric inequality and its associated stability inequality; we refer to the introduction of \cite{DoEsFiFrLo} for more details and references about this case. 

We finally mention that the recent preprint \cite{ChLuTa} gives an abstract condition under which the global best constant of a general stability inequality is equal to the local best constant corresponding to the right side of \eqref{spectral bound}. However, consistently with Theorem \ref{theorem strict inequality}, this condition is not fulfilled for the Sobolev stability inequality \eqref{bianchi-egnell}.

\section{Proof of Theorem \ref{theorem strict inequality}}

We prove Theorem \ref{theorem strict inequality} by using the same idea as for inequality \eqref{spectral bound}, but we manage to be more precise in the asymptotic expansion near $\mathcal M$ and in the choice of the perturbation $\rho$. 

To be more precise, let us introduce some more notation. We fix the standard Talenti bubble 
\[ U(y) := (1 + |y|^2)^{-\frac{d-2s}{2}}. \]
and denote its dilated and translated version by 
\[ U_\xl(y) := \lambda^\frac{d-2s}{2} U(\lambda (y-x)).  \]
Then the tangent space $T$ of $\mathcal M$ at $U$ is spanned by $U$ and the $d+1$ functions
\[ V_0 := \partial_\lambda|_{\lambda = 1} U_{0, \lambda}, \qquad V_i:= \pxi|_{x = 0} U_{x,1} . \qquad i = 1,...,d \]
We denote by $T^\perp$ the orthogonal of $T$ with respect to the $\dot{H}^s$ scalar product $\langle u, v \rangle_{\dot{H}^s(\R^d)} = \int_{\R^d}(-\Delta)^{s/2} u (-\Delta)^{s/2} v \, dx$.

We also introduce the (inverse) stereographic projection $\mathcal S: \R^d \to \Sph^d$ given, in Cartesian coordinates of $\Sph^d \subset \R^{d+1}$, by 
\begin{equation}
\label{ster proj definition}
(\mathcal S(x))_i = \frac{2 x_i}{1 + |x|^2} \quad (i = 1,...,d), \qquad (\mathcal S(x))_{d+1} = \frac{1 - |x|^2}{1+|x|^2}. 
\end{equation} 
It is convenient to denote by $J_{\mathcal S}(x)= |\det D \mathcal S(x)| = \left(\frac{2}{1 + |x|^2}\right)^d = 2^d U(x)^{2^*}$ its Jacobian.  The relevant transformation properties of the stereographic projection and discussed to some greater extent, e.g.,  in \cite[Section 2]{ChFrWe}.

The following spectral gap inequality plays a crucial role in \cite{BiEg} and \cite{ChFrWe} and seems to appear for the first time in \cite{Rey}, for $s = 1$. For $s \in (0, \frac{d}{2})$, a detailed statement and proof can be found in \cite{DeKo}. 

\begin{proposition}
\label{proposition spectral gap ineq}
Let $\rho \in T^\perp$. Then 
\[ \|(-\Delta)^{s/2} \rho\|_2^2 - S_{d,s} (2^*-1) \|U\|_{2^*}^{2 - 2^*} \int_{\R^d} U^{2^* - 2} \rho^2 \, dx \geq \frac{4s}{d+2s+2} \|(-\Delta)^{s/2} \rho\|_2^2. \]
Moreover, equality holds if and only if 
\begin{equation}
\label{rey equality case}
\rho(x) = J_\mathcal S(x)^\frac{1}{2^*} v_2(\mathcal S(x)), 
\end{equation} 
with $v_2$ a spherical harmonic of degree $\ell = 2$ (i.e. the restriction to $\Sph^d$ of a homogeneous harmonic polynomial on $\R^{d+1}$ of degree $2$). 
\end{proposition}

We can now prove our main result.

\begin{proof}
[Proof of Theorem \ref{theorem strict inequality}]
We build a sequence of test functions of the form 
\[ f_\eps = U + \eps \rho, \]
where $\eps \to 0$ and  $\rho \in T^\perp$ is as in \eqref{rey equality case}, with a certain spherical harmonic $v_2$ to be determined. 

The orthogonality relations, and the fact that $(-\Delta)^{s/2} U = c_{d,s} U^{2^* -1}$, for $c_{d,s} = \frac{S_{d,s}}{\|U\|_{2^*}^{2^* - 2}}$, easily imply 
\[ \|f_\eps\|_2^2 = \|(-\Delta)^{s/2} U\|_2^2 + \eps^2 \|\rho\|_2^2 \]
and 
\[ \int_{\R^d} U^{2^* - 1} \rho \, dx = \frac{1}{c_{d,s}} \int_{\R^d} (-\Delta)^{s/2} U  (-\Delta)^{s/2} \rho \, dx = 0. \]
On the other hand, a Taylor expansion yields 
\begin{align*}
(U + \eps \rho)^{2^*} = U^{2^*} + \eps {2^*} U^{{2^*}-1} \rho + \eps^2 \frac{{2^*}({2^*}-1)}{2} U^{{2^*}-2} \rho^2 + \eps^3 \frac{{2^*}({2^*}-1)({2^*}-2)}{6} U^{{2^*}-3} \rho^3 + o(\eps^3), 
\end{align*} 
Notice that the Taylor expansion up to third order is justified no matter the value of ${2^*}$, because $\eps |\rho(x)| << U(x)$ in every point $x \in \R^d$, as $\eps \to 0$. 

Hence, using $\int_{\R^d} U^{{2^*}-1} \rho \, dx = 0$, 
\begin{align*}
&\quad \|U + \eps \rho\|_{2^*}^2 \\
&= \left( \int_{\R^d} \left(U^{2^*} + \eps^2 \frac{{2^*}({2^*}-1)}{2} U^{{2^*}-2} \rho^2 + \eps^3 \frac{{2^*}({2^*}-1)({2^*}-2)}{6} U^{{2^*}-3} \rho^3 \right) \, dx \right)^\frac{2}{{2^*}}  + o(\eps^3) \\
&= \|U\|_{2^*}^2 + \eps^2 ({2^*}-1)  \|U\|_{2^*}^{2-{2^*}} \int_{\R^d} U^{{2^*}-2} \rho^2 \, dx \\
&\quad  + \eps^3 \frac{({2^*}-1)({2^*}-2)}{3} \|U\|_{2^*}^{2-{2^*}} \int_{\R^d} U^{{2^*}-3}  \rho^3 \, dx  
   + o(\eps^3). 
\end{align*}
Now recall that $\|(-\Delta)^{s/2} U\|_2^2 = S_{d,s} \|U\|_{2^*}^2$ and that we have chosen $\rho$ to achieve equality in Proposition \ref{proposition spectral gap ineq}. We thus find that the numerator of the Bianchi--Egnell quotient in \eqref{bianchi-egnell} equals
\begin{align*}
&\quad \|(-\Delta)^{s/2} f_\eps\|_2^2 - S_{d,s} \|f_\eps\|_{2^*}^2 \\
&= \left(\|(-\Delta)^{s/2} U\|_2^2 - S_{d,s} \|U\|_{2^*}^2\right) + \eps^2 \left(\|(-\Delta)^{s/2} \rho\|_2^2 - S_{d,s} (2^* - 1) \|U\|_{2^*}^{2 - 2^*} \int_{\R^d} U^{2^*-2} \rho^2 \, dx \right) \\
& \quad - \eps^3 S_{d,s} \frac{(2^*-1)(2^*-2)}{3} \|U\|_{2^*}^{2 - 2^*} \int_{\R^d} U^{2^* - 3} \rho^3 \, dx + o(\eps^3) \\
&= \frac{4}{d +2s + 2} \eps^2 \|(-\Delta)^{s/2} \rho\|_2^2 - \eps^3 S_{d,s} \frac{(2^*-1)(2^*-2)}{3} \|U\|_{2^*}^{2 - 2^*} \int_{\R^d} U^{2^* - 3} \rho^3 \, dx + o(\eps^3). 
\end{align*}
It follows by the implicit function theorem (see \cite{BaCo} for a similar argument for $s=1$, which remains valid for all $s \in (0, \frac{d}{2})$ as observed in \cite{AbChHa}) that for $\eps > 0$ small enough the minimum of the distance $\dist(U + \eps \rho, \mathcal M)$ is in fact achieved in $U$. Hence
\[ \dist(U + \eps \rho, \mathcal M)^2 = \eps^2 \|(-\Delta)^{s/2} \rho\|^2_2. \]
Thus the above expansions yield, for $\eps > 0$ small enough, the desired strict inequality 
\[ \mathcal E(U + \eps \rho) < \frac{4s}{d+2s+2}, \]
provided we can choose $v_2$ in a way such that (with the relation \eqref{rey equality case} between $v_2$ and $\rho$)
\begin{equation}
\label{int rho^3 strictly positive}
\int_{\R^d} U^{2^* - 3} \rho^3 \, dx > 0. 
\end{equation}

To achieve this, we make the choice 
\[ v_2(\omega) = \omega_1 \omega_2 + \omega_2 \omega_3 + \omega_3 \omega_1, \]
which is clearly a spherical harmonic of degree $\ell = 2$. (This choice of $v_2$ is what necessitates the additional hypothesis $d \geq 2$.) We have 
\[ v_2(\omega)^3 = I_1(\omega) + I_2(\omega) + I_3(\omega), \]
where 
\begin{align*}
I_1(\omega) &= 6 \omega_1^2 \omega_2^2 \omega_3^2, \\
I_2 (\omega) &= 3 ( \omega_1 \omega_2^2 \omega_3^3 + \omega_1 \omega_3^2 \omega_2^3 + \omega_2 \omega_1^2 \omega_3^3 + \omega_3 \omega_2^2 \omega_1^3 + \omega_2 \omega_3^2 \omega_1^3 + \omega_3 \omega_1^2 \omega_2^3), \\
I_3( \omega) &= \omega_1^3 \omega_2^3 + \omega_1^3 \omega_3^3 + \omega_2^3 \omega_3^3. 
\end{align*} 
 Writing \eqref{rey equality case} as $\rho(x) = 2^{\frac{d-2s}{2}} U(x) v_2(\mathcal S(x))$, we get 
\begin{align*}
\int_{\R^d} U^{2^* - 3} \rho^3 \, dx &= 2^{\frac{3(d-2s)}{2} - d} \int_{\Sph^d} v_2(\omega)^3 \, d\omega =  2^{\frac{3(d-2s)}{2} - d} \int_{\Sph^d} (I_1(\omega) + I_2(\omega) + I_3(\omega))  \, d \omega \\
& = 6 \times  2^{\frac{3(d-2s)}{2} - d} \int_{\Sph^d} \omega_1^2 \omega_2^2 \omega_3^2 \, d\omega > 0, 
\end{align*}
because all the monomials in $I_2$ and $I_3$ are odd functions of at least one coordinate $\omega_i$ and thus cancel in the integral. This completes the proof of Theorem \ref{theorem strict inequality}.
\end{proof}

\end{document}